\IfFileExists{conm-p-l.cls}{
\documentclass{conm-p-l}
}{
\documentclass[12pt]{amsart}
\usepackage{fullpage}
}
\RequirePackage[l2tabu, orthodox]{nag}
\usepackage{enumitem,url,amssymb,colonequals}
\usepackage[all]{xy} % for commutative diagrams
\usepackage{mathrsfs} % for \mathscr (script letters)https://www.overleaf.com/project/5e73c09cf2cf190001da6dbf
\usepackage[dvipsnames]{xcolor}
\usepackage{tikz-cd}
\usepackage{wasysym}
\usepackage{hyperref}
\hypersetup{colorlinks=true,urlcolor=blue,citecolor=blue,linkcolor=blue}
\usepackage{cleveref}

\usepackage[style=alphabetic, backend=biber, sorting=nyt, doi=false, isbn=false,url=false, hyperref,backref,backrefstyle=none, maxnames=20, maxalphanames=20]{biblatex}
\newbibmacro{string+doiurlisbn}[1]{%
  \iffieldundef{doi}{%
    \iffieldundef{url}{%
      \iffieldundef{isbn}{%
        \iffieldundef{issn}{%
          #1%
        }{%
          \href{https://books.google.com/books?vid=ISSN\thefield{issn}}{#1}%
        }%
      }{%
        \href{https://books.google.com/books?vid=ISBN\thefield{isbn}}{#1}%
      }%
    }{%
      \href{\thefield{url}}{#1}%
    }%
  }{%
    \href{https://dx.doi.org/\thefield{doi}}{#1}%
  }%
}
\DeclareFieldFormat{title}{\usebibmacro{string+doiurlisbn}{\mkbibemph{#1}}}
\DeclareFieldFormat[article,incollection,inproceedings,unpublished,misc,book]{title}%
    {\usebibmacro{string+doiurlisbn}{\mkbibquote{#1}}}
\DefineBibliographyStrings{english}{%
  backrefpage = {$\uparrow$},%
  backrefpages = {$\uparrow$}%
}
\usepackage[OT2,T1]{fontenc}
\DeclareSymbolFont{cyrletters}{OT2}{wncyr}{m}{n}
\DeclareMathSymbol{\Sha}{\mathalpha}{cyrletters}{"58}

\title{Curve equations from expansions of 1-forms at a nonrational point}
\author[R. van Bommel]{Raymond van Bommel}
\address{
  Department of Mathematics,
  Massachusetts Institute of Technology,
  Cambridge,
  MA 02139,
  USA
}
\address{School of Mathematics,
University of Bristol,
Fry Building,
Woodland Road,
Bristol,
BS8 1UG,
UK}
\email{\href{mailto:r.vanbommel@bristol.ac.uk}{r.vanbommel@bristol.ac.uk}}
\urladdr{\url{https://raymondvanbommel.nl/}}

\author[E. Costa]{Edgar Costa}
\address{
  Department of Mathematics,
  Massachusetts Institute of Technology,
  Cambridge,
  MA 02139,
  USA
}
\email{\href{mailto:edgarc@mit.edu}{edgarc@mit.edu}}
\urladdr{\url{https://edgarcosta.org}}

\author[B. Poonen]{Bjorn Poonen}
\address{
  Department of Mathematics,
  Massachusetts Institute of Technology,
  Cambridge,
  MA 02139,
  USA
}
\email{\href{mailto:poonen@math.mit.edu}{poonen@math.mit.edu}}
\urladdr{\url{https://math.mit.edu/~poonen/}}

\author[P.~Srinivasan]{Padmavathi Srinivasan}
\address{Boston University, 665 Commonwealth Avenue, Boston, MA 02215, USA}
\email{\href{mailto:padmask@bu.edu}{padmask@bu.edu}}
\urladdr{\url{https://padmask.github.io/}}

\date{June 10, 2025}

\newenvironment{enumalph}
{\begin{enumerate}}
{\end{enumerate}}

\newcommand{\defi}[1]{\textsf{#1}} 	% for defined terms

\newcommand{\C}{\mathbb{C}}

\newcommand{\PP}{\mathbb{P}}
\newcommand{\Q}{\mathbb{Q}}

\newcommand{\Z}{\mathbb{Z}}
\newcommand{\Qbar}{{\overline{\Q}}}

\newcommand{\calR}{\mathcal{R}}
\newcommand{\calS}{\mathcal{S}}
\newcommand{\calT}{\mathcal{T}}

\newcommand{\OO}{\mathscr{O}}

\DeclareMathOperator{\Div}{Div}

\DeclareMathOperator{\ord}{ord}

\DeclareMathOperator{\Proj}{Proj}

\DeclareMathOperator{\Sym}{Sym}

\DeclareMathOperator{\Tr}{Tr}

\newcommand{\GL}{\operatorname{GL}}

\newcommand{\HH}{{\operatorname{H}}}

\newcommand{\del}{\partial}
\newcommand{\intersect}{\cap} % binary intersection
\newcommand{\isom}{\simeq}

\newcommand{\tensor}{\otimes} % binary tensor product
\newtheorem{theorem}{Theorem}[section]
\newtheorem{lemma}[theorem]{Lemma}
\newtheorem{corollary}[theorem]{Corollary}

\newtheorem*{proposition*}{Proposition*}

\theoremstyle{definition}

\theoremstyle{remark}
\newtheorem{remark}[theorem]{Remark}

\begin{document}

\begin{abstract}
We exhibit an algorithm to compute equations of an algebraic curve over a computable characteristic~0 field from the power series expansions of its regular 1-forms at a \emph{nonrational} point of the curve, extending a 2005 algorithm of 
Baker, Gonz\'{a}lez-Jim\'{e}nez, Gonz\'{a}lez, and Poonen for expansions at a rational point.
If the curve is hyperelliptic, the equations present it as an explicit double cover of a smooth plane conic, or as a double cover of the projective line when possible.
If the curve is nonhyperelliptic, the equations cut out the canonical model.
The algorithm has been used to compute equations over $\Q$ for many hyperelliptic modular curves without a rational cusp in the L-functions and Modular Forms Database.
\end{abstract}

\maketitle

\section{Introduction}

A curve $X$ is called \defi{nice} if it is smooth, projective, and geometrically integral.
From now on, $X$ is a nice curve of genus $g \ge 2$ over $\Q$,
but all our theorems and algorithms work over any ground field $F$ of characteristic~$0$ if field operations in $F$ are computable.
Our goal is to give an algorithm that takes as input the initial terms of the expansions of $1$-forms forming a basis of the $\Q$-vector space $\Gamma(X,\Omega^1_{X/\Q})$ 
at a \emph{nonrational} point and returns equations for $X$ over $\Q$;
such input arises naturally in \cite[Section~5]{zywina-openimage}, for instance.
(The analogue for expansions at a \emph{rational} point is covered in \cite[Section~2.1]{baker-gonzalez-jimenez-gonzalez-poonen-05}.)
These equations will cut out the canonical model if $X$ is nonhyperelliptic, and a double cover of a genus $0$ curve if $X$ is hyperelliptic; see \Cref{theorem:recovering} for more details.
Examples of the nonhyperelliptic case (the easier case) were worked out in 
\cite[Section~7]{baran} and \cite[Sections~7~and~8]{mercuri-schoof}.
So the main new work is in the hyperelliptic case.

The algorithm in \cite[Section~2.1]{baker-gonzalez-jimenez-gonzalez-poonen-05} will produce a model over the field of definition of the nonrational point, but there is no easy way to pass from that to the equation over $\Q$. 
Also, the presence of the rational point in \cite[Section~2.1]{baker-gonzalez-jimenez-gonzalez-poonen-05} meant that in the hyperelliptic case, the image of the canonical map was $\PP^1_\Q$, whereas in the present article, it could instead be a pointless genus~$0$ curve instead of $\PP^1_\Q$, in which case $X$ will need to be given as a double cover of a plane conic.
Moreover, there are additional complications in our article coming from the fact that even the expansion of objects defined over $\Q$ have coefficients in a larger number field.

The motivation for our algorithm is the problem of finding equations of modular curves \emph{that have no rational cusp}.
The algorithm has been used so far to calculate equations of over 4700 such hyperelliptic modular curves without a rational cusp for the L-Functions and Modular Forms Database \cite{lmfdb}, among which over 1500 are a double cover of a pointless genus 0 curve.
\iffalse
# these are the number of modular curves with a Geometric Weierstrass model
sage: db.modcurve_models.count({'model_type':7})
1512
sage: labelsgh = list(db.modcurve_models.search({'model_type':7}, 'modcurve'))
sage: db.gps_gl2zhat_fine.count({'label' : {'$in': labelsgh}, 'rational_cusps': 0})
1512
sage: set(db.gps_gl2zhat_fine.search({'label' : {'$in': labelsgh}, 'rational_cusps': 0 }, 'genus'))
{3, 5}

# hyperellitic curves y^2 = F
sage: db.modcurve_models.count({'model_type':5})
5603
sage: labelshyp = list(db.modcurve_models.search({'model_type':5}, 'modcurve'))
sage: db.gps_gl2zhat_fine.count({'label' : {'$in': labelshyp}, 'rational_cusps': 0})
3218
sage: set(db.gps_gl2zhat_fine.search({'label' : {'$in': labelshyp}, 'rational_cusps': 0 }, 'genus'))
{1, 2, 3, 4, 5, 7, 11}
sage: db.gps_gl2zhat_fine.count({'label' : {'$in': labelshyp}, 'rational_cusps': {'$gt' : 0 }})
2385
sage: set(db.gps_gl2zhat_fine.search({'label' : {'$in': labelshyp}, 'rational_cusps': {'$gt': 0} }, 'genus'))
{1, 2, 3, 4, 5, 6, 7, 11}
\fi

\section{Hyperelliptic curves}

The curve $X$ is called \defi{hyperelliptic} if the canonical map $X \to \PP^{g-1}$ is not a closed immersion.
Equivalently, $X$ is hyperelliptic if there exists a degree~$2$ morphism $\pi$ from $X$ to some genus~$0$ curve $C$.
Suppose that this is the case.
Then $C$ and the morphism $\pi$ are unique up to isomorphism. 
In fact, $C$ is the image of the canonical map.
The curve $C$ need not be isomorphic to $\PP^1$ over $\Q$, but the anticanonical map for $C$ identifies $C$ with a smooth plane conic in $\PP^2 = \Proj \Q[a,b,c]$.
For any $d \in \Z_{\ge 0}$, let $\Q[a,b,c]_d$ be the space of degree $d$ homogeneous polynomials in $\Q[a,b,c]$.

\section{Main theorem}

Now return to the general case, in which $X$ is any nice curve of genus $g \ge 2$ over $\Q$. 
Let $K \supseteq \Q$ be a finite extension.
Let $X_K = X \times_\Q K$.
Let $P \in X(K)$.
Assume that $K=\Q(P)$.
Let $q$ be a uniformizer of the completed local ring $\widehat{\OO}_{X_K,P}$.
Let $\omega_1,\ldots,\omega_g$ be a $\Q$-basis of $\HH^0(X,\Omega^1)$.
For each $i \in \{1,\ldots,g\}$, the Taylor expansion of $\omega_i$ at $P$ is $w_i \, dq$ for some $w_i \in K[[q]]$.
For $B \in \Z_{>0}$,
let $K[q]_{<B} \isom K[[q]]/(q^B)$ be the vector space of polynomials of degree $<B$.
Let $\bar{w}_i \colonequals (w_i \bmod q^B) \in K[q]_{<B}$.

\begin{theorem}\label{theorem:recovering}
Let $B=19g+48$. 
There exists an algorithm with 

Input: $g$, $K$, and polynomials $\bar{w}_1,\ldots,\bar{w}_g \in K[q]_{<B}$ arising from some nice curve $X$ over $\Q$ and $P \in X(K)$ as above.

Output: 
\begin{itemize}
\item If $X$ is nonhyperelliptic, return \texttt{nonhyperelliptic} and a finite list of homogeneous polynomials over $\Q$ cutting out a curve in $\PP^{g-1}$ linearly isomorphic over $\Q$ to the canonical model of $X$.
\item If $X$ is hyperelliptic and $g$ is even, return \texttt{hyperelliptic} and a separable polynomial $f \in \Q[x]$ of degree $2g+1$ or $2g+2$ such that $X$ is birational to the curve $y^2=f(x)$.
\item If $X$ is hyperelliptic and $g$ is odd, return \texttt{hyperelliptic} and homogeneous polynomials $Q \in \Q[a,b,c]_2$ and $H \in \Q[a,b,c]_{g+1}$ such that 
\[ 
    C \isom \Proj \frac{\Q[a,b,c]}{(Q)} \subset \PP^2 \quad\textup{ and }\quad X \isom \Proj \frac{\Q[a,b,c,y]}{(y^2-H,Q)} \subset \PP \left(1,1,1,\frac{g+1}{2} \right).
\] 
In this case, if a rational point on $C$ is given, find a model $y^2=f(x)$ as in the even genus hyperelliptic case.
% \item If $X$ is hyperelliptic and $g$ is odd, return \texttt{hyperelliptic} and a degree~$2$ homogeneous polynomial in $\Q[a,b,c]$ defining a smooth plane conic in $\PP^2$ isomorphic to $C$, and a ratio of homogeneous polynomials in $\Q(a,b,c)$ that defines a rational function $h \in \Q(C)$ such that $\Q(X) \isom \Q(C)(\sqrt{h})$.
\end{itemize}
%In the hyperelliptic case, one can also check whether $C \isom \PP^1$, and if so, find a polynomial $f \in \Q[x]$ such that $X$ is birational to the curve $y^2=f(x)$.
\end{theorem}

\begin{remark}
In the odd genus hyperelliptic case, we may require the quadratic form $Q$ to be \emph{diagonal}, if desired.
\end{remark}

\begin{remark}
By computing Hilbert symbols, one can determine whether a given smooth plane conic $C$ over $\Q$ is isomorphic to $\PP^1_\Q$; this is essentially due to Legendre.
(More generally, by the Hasse--Minkowski local--global principle for quadratic forms, this can be done over any number field; see, e.g., \cite[Theorem~26.3]{shimura}.
But it involves more than just field operations,
so it is not an algorithm that generalizes to \emph{any} characteristic $0$ field.)
\end{remark}

%\begin{remark}
%Our proof of Theorem~\ref{theorem:recovering} works for an arbitrary extension of fields $K \supset F$ in place of $K \supset \Q$; we stick with the latter to simplify notation.  The case where $K=F$ (in characteristic~$0$) has been covered in \cite[Section 2.1]{baker-gonzalez-jimenez-gonzalez-poonen-05}.
%\end{remark}

\section{Theoretical lemmas}

Before explaining the algorithm, we prove a few theoretical lemmas.
Let $S = \Q[x_1,\ldots,x_g]$ be the homogeneous coordinate ring of $\PP^{g-1}$ over $\Q$.
Let $I \subset S$ be the homogeneous ideal of the canonical image of $X$.
Let $I_d \subset S_d$ be the degree $d$ parts of $I \subset S$.

\begin{lemma} \label{lemma:vanishing}
Let $f \in S_d$.
If $f(w_1,\ldots,w_g) \in K[[q]]$ vanishes at $q=0$ to order $> d(2g-2)/[K:\Q]$, then the corresponding section of $(\Omega^1)^{\tensor d}$ is $0$.
\end{lemma}

\begin{proof}
The section has more than $d(2g-2) = \deg (\Omega^1)^{\tensor d}$ geometric zeros in total (at $P$ and its conjugates), so it is $0$.
\end{proof}

\begin{corollary}\label{corollary:Idbasis}
If $B > d(2g-2)/[K:\Q]$, then from the input as in Theorem~\ref{theorem:recovering} one can compute a basis for $I_d$.
\end{corollary}

\begin{proof}
By \Cref{lemma:vanishing}, 
$I_d$ is the kernel of the $\Q$-linear map $S_d \to K[[q]]/(q^B)$ sending each monomial to its truncated expansion.
\end{proof}

\begin{lemma}\label{lemma:I2_hyp}
The dimension of $I_2$ is $\binom{g-1}{2}$ if $X$ is hyperelliptic, and $\binom{g-2}{2}$ if not.
\end{lemma}

\begin{proof}
We may work over $\C$.
Let $\C[x]_{\le n}$ be the space of polynomials of degree at most $n$.
In the hyperelliptic case, $X$ is the smooth projective model of $y^2=F(x)$ with $\deg F = 2g+1$, and $H^0(X,\Omega^1) = \C[x]_{\le g-1} \frac{dx}{y}$ (see \cite[p.~11]{ACGH}, for example), so $I_2$ is isomorphic to the kernel of the surjective map $\ker(\Sym^2 \C[x]_{\le g-1} \to \C[x]_{\le 2g-2})$, so $\dim I_2 = g(g+1)/2 - (2g-1) = \binom{g-1}{2}$.
In the nonhyperelliptic case, this follows from Max Noether's theorem \cite[p.~117]{ACGH}.
\end{proof}

\begin{lemma}
\label{lemma:findrat}
Let $X$ be a hyperelliptic curve over $\Q$.
Let $L$ be a finite extension of $\Q$.
Let $P' \in X(L)$.
Suppose that $\omega_1',\ldots,\omega_g'$ is an $L$-basis for $\HH^0(X_L,\Omega^1)$ such that $\ord_{P'}(\omega_1') < \ldots < \ord_{P'}(\omega_g')$.
Let $t = \omega_{g-1}'/\omega_g' \in L(X)$.
Then $t \in L(C)$ and is of degree $1$ $($as a rational function on $C_L$$)$.
\end{lemma}

\begin{proof}
We may assume that $X_L$ is the smooth projective model of $y^2=F(x)$ for some $F \in L[x]$, and $P'$ is at infinity.
Then for $i=0,\ldots,g-1$, we have $\omega_{g-i}' = J_i(x) \, dx/y$ for some $J_i(x) \in L[x]$ of degree exactly $i$.
Then $t$ is a degree~$1$ polynomial in $L[x]$.
% The hyperelliptic involution $\iota$ acts as $-1$ on each $\omega_i'$, and hence as $+1$ on $t$. Thus $t \in L(X)^\iota = L(C)$.
\end{proof}

\begin{lemma}
\label{lemma:anticanonical system on C}
Let $C$ be a genus~$0$ curve over $\Q$.
Let $\calT$ be the tangent bundle of $C$.
Let $V \colonequals \HH^0(C,\calT)$.
Let $L$ be a finite extension of $\Q$, with $\Q$-basis $\lambda_1,\ldots,\lambda_\ell$.
Let $t$ be a degree~$1$ rational function on $C_L$.
\begin{enumalph}
\item The meromorphic sections
$\frac{d}{dt}, t \frac{d}{dt}, t^2 \frac{d}{dt}$ of $\calT$ form an $L$-basis of $V_L$.
\item\label{brat}\label{spanningtraces} The elements $\Tr_{L/\Q}(\lambda_j t^i \frac{d}{dt})$ for $0 \le i \le 2$ and $1 \le j \le \ell$ span $V$.
\end{enumalph}
\end{lemma}

\begin{proof}\hfill
\begin{enumalph}
\item
Without loss of generality, $C_L=\PP^1$ and $t$ is the standard coordinate.
Then $\calT \isom \OO(2)$, so $\dim V_L=3$.
Also, $\frac{d}{dt}$ has a double zero at $\infty$, so $\frac{d}{dt}, t \frac{d}{dt}, t^2 \frac{d}{dt}$ are independent global sections.
\item
The map $\Tr_{L/\Q} \colon V_L \to V$ is surjective.\qedhere
\end{enumalph}
\end{proof}

\begin{lemma}
\label{lemma:expressrat}
Let $C$ be a smooth plane conic in $\PP^2$ over a field $k$.
Let $h \in k(C)$ be a rational function of degree $d$.
Then $h$ is given by a ratio of two homogeneous forms on $\PP^2$ of degree $\lceil d/2 \rceil$.
\end{lemma}
\begin{proof}
Let $L \in \Div C$ be a hyperplane section of $C \subset \PP^2$.  Write $(h) = (h)_0 - (h)_\infty$, where $(h)_0$ and $(h)_\infty$ are effective and of degree $d$.
Then $\lceil d/2 \rceil L - (h)_\infty$ is of degree $\ge 0$, so by Riemann--Roch there exists a section $s$ of $\OO_C(\lceil d/2 \rceil)$ vanishing at the poles of $h$. 
Then $hs$ is another global section of $\OO_C(\lceil d/2 \rceil)$.
Both $s$ and $hs$ are restrictions of homogeneous forms on $\PP^2$, and $h$ is their ratio.
\end{proof}

\begin{lemma}
\label{L:pullback of sections}
Let $\pi \colon X \to Y$ be a morphism of nice curves over $\C$.
Let $P \in X(\C)$.
Let $Q=\pi(P)$.
Let $e$ be the ramification index of $\pi$ at $P$.
Let $s$ be a nonzero meromorphic section of $(\Omega^1_Y)^{\tensor n}$ for some $n \in \Z$.
Then $\ord_P (\pi^*s) = e \ord_Q s + n(e-1)$.
\end{lemma}

\begin{proof}
Let $t$ be a uniformizer at $\pi(P)$ on $Y$.
For any $f \in \Q(Y)^\times$, we have $\ord_P(\pi^* f) = e \ord_Q(f)$ by definition, and $\ord_P(\pi^* dt) = e-1$ as in the proof of the Hurwitz formula.
Since $s = f \, dt^{\tensor n}$ for some $f \in \Q(Y)^\times$, the formula follows.
\end{proof}

\section{Proof of main theorem ignoring precision}
\label{S:proof of main theorem}

We now start the proof of \Cref{theorem:recovering}.
Compute a basis for $I_2$ using \Cref{corollary:Idbasis} and apply \Cref{lemma:I2_hyp} to test if $X$ is hyperelliptic.
If $X$ is nonhyperelliptic, compute bases for $I_2,I_3,I_4$ using \Cref{corollary:Idbasis};
these are enough to cut out $X \subset \PP^{g-1}$, by Petri's theorem \cite{petri}.
\emph{Henceforth, we assume that $X$ is hyperelliptic.}

Steps~\eqref{Step:ConvenientBasis}--\eqref{Step:MakeGlobalSections} below require working over a field $L$ such that $X$ has an $L$-point $P'$, so that there is an isomorphism $C_L \isom \PP^1_L$ such that $P'$ maps to $\infty$.
We choose $L$ to be an isomorphic copy of $K$, and let $P' \in X(L)$ be the result of applying the isomorphism to $P \in X(K)$.
(We will need to consider $(L \tensor K)/K$-traces, so keeping separate names for $L$ and $K$ will help clarify things.)
% (Our actual choice will be $(L,P')=(K,P)$, but we keep notation separate to clarify roles.)
We will need to take $L/\Q$-traces of elements of $\HH^0(C_L,\calT_L)$ as in \Cref{lemma:anticanonical system on C} to get elements of $\HH^0(C,\calT)$; 
these will be computed as $(L \tensor K)/K$-traces of their expansions at $P$ (the tensor product is over $\Q$).

We first explain the algorithm as if we had $w_1,\ldots,w_g \in K[[q]]$ to infinite precision, and later in \Cref{S:Precision} explain what modifications are needed when we have only their truncations $\bar{w}_1,\ldots,\bar{w}_g$.

\begin{enumerate}
\item\label{Step:ConvenientBasis} 
(Find the expansion of a rational function $t \colon X_L \to C_L \isom \PP^1_L$.)
Let $W \subset K[[q]]$ be the $\Q$-span of $w_1,\ldots,w_g$,
and let $W_K \subset K[[q]]$ be their $K$-span.
Run Gaussian elimination over $K$ to find a new $K$-basis $w_1',\ldots,w_g'$ of $W_K$ such that $\ord_P(w_1')<\cdots<\ord_P(w_g')$.
Let $M \in \GL_g(K)$ be the change-of-basis matrix sending $w_1,\ldots,w_g$ to $w_1',\ldots,w_g'$.
Applying the isomorphism $K \to L$ yields a matrix $M_L \in \GL_g(L)$.
Then $M_L$ sends $\omega_1,\ldots,\omega_g$ to an $L$-basis $\omega_1',\ldots,\omega_g'$ of $\HH^0(X_L,\Omega^1)$ as in \Cref{lemma:findrat}.
Computing the same $L$-linear combinations of $w_1,\ldots,w_g \in K[[q]]$ produces elements $w_1'',\ldots,w_g'' \in L \tensor W \subset (L \tensor K)[[q]]$ representing the expansions at $P$ of the $\omega_i'$, which have increasing order of vanishing at $P'$.

Let $t = \omega_{g-1}'/\omega_g' \in L(X)$, as in \Cref{lemma:findrat},
so $t$ is the ``$x$-coordinate'' on a hyperelliptic model.
Its expansion at $P$ is in $(L \tensor K)((q))$.

\item\label{Step:MakeGlobalSections} 
(Find expansions of a $\Q$-basis of $\HH^0(C,\calT)$.)
Let $\lambda_1,\ldots,\lambda_\ell$ be a $\Q$-basis of $L$.
The $L/\Q$-traces in \Cref{lemma:anticanonical system on C}\eqref{brat} span $V \colonequals \HH^0(C,\calT)$, so three of them form a $\Q$-basis of $C$.
To calculate with them, we start with the expansions of $\lambda_j t^i \frac{d}{dt}$ in $(L \tensor K)((q))$ for $i = 0,1,2$ and $j=1,\ldots,\ell$,
calculate $(L\tensor K)/K$-traces (traces are compatible with base change),
and find three of them that are $K$-linearly independent and hence $\Q$-linearly dependent;
call them $\del_0,\del_1,\del_2 \in K((q)) \frac{d}{dq}$;
these are the expansions of a basis of global sections of $\calT$ pulled back to $X$.

\item\label{Step:FindingQ}
(Find the equation $Q=0$ of the conic $C$.) 
There is a unique $Q \in \Q[a,b,c]_2$ up to scalar
such that $Q(\del_0,\del_1,\del_2)=0$ in $K((q)) \left( \frac{d}{dq} \right)^2$.
Then $Q=0$ is the anticanonical model of $C$ in $\PP^2$.
We find $Q$ by linear algebra.

\item\label{Step:SquareRoot}
(Find the expansion of $h \in \Q(C)$ such that $\Q(X) = \Q(C)(\sqrt{h})$.)
In this step, we compute $h \in \Q(C)$ such that $\Q(X) = \Q(C)(\sqrt{h})$.
Let $f \colonequals a/b$, viewed as a rational function on $C$; its expansion is $\del_0/\del_1$. 
Let $y=df/\omega_1 \in \Q(x)$ and $h=y^2$.
The hyperelliptic involution fixes $df$ and acts as $-1$ on $H^0(X, \Omega^1)$, so it negates $y$ and fixes $h$; that is, $h \in \Q(C)$.
Then $\Q(X) = \Q(C)(y) = \Q(C)(\sqrt{h})$.

\item\label{Step:Writinghasratio}(Write $h$ as a ratio of homogeneous forms.)
We now show how to write $h$ explicitly as $F/G$ for some $F,G \in \Q[a,b,c]_{g+3}$. Since $f$ is a rational function of degree $2$ on $C$, it has at most $2$ poles with multiplicity, so $df$ on $C$ has at most $4$ poles with multiplicity (the worst case being when $f$ has two simple poles), so its pullback to $X$ has at most $8$ poles. 
On the other hand, $\omega_1$ has at most $2g-2$ zeros on $X$, so $y$ has at most $2g+6$ poles on $X$.
Then $h$ has at most $2(2g+6)$ poles on $X$, so its degree on $C$ is at most $2g+6$. 
By \Cref{lemma:expressrat}, there exist homogeneous forms $F,G \in \Q[a,b,c]_{g+3}$ such that $F/G = h$. 
To find the coefficients of possible $F$ and $G$, we solve the linear system $F = h G$ in these unknown coefficients, using expansions of $\del_0,\del_1,\del_2$ and $h$.

\item\label{Step:EvenGenus} 
(For even $g$, find an equation $y^2=f(x)$ for $X$.)
Suppose that $g$ is even.
In this case, $C \isom \PP^1$, and we will describe a method to find a rational parametrization of $C$, following the strategy of \Cref{lemma:expressrat}.
The $1$-form $\omega_1$ corresponds to a linear form on $\PP^{g-1}$, which cuts out a divisor $D$ of odd degree $g-1$ on $C$.
Let $\calS$ be the space of $S \in \Q[a,b,c]_{g/2}$ that vanish along $D$.
By the Riemann--Roch theorem, $\dim \calS=2$;
we next seek an explicit basis of $\calS$, which will define an isomorphism $C \to \PP^1$.
For each $S \in \calS$ and for $j=2,\ldots,g$, the element $R_j \colonequals S \omega_j/\omega_1 \in \Q(a,b,c)$ lies in $\Q[a,b,c]_{g/2}$ since $S$ vanishes along $D$.
Thus $\calS$ is the projection on the last coordinate of the space $\calR$ of $g$-tuples $(R_2,\ldots,R_g,S)$ of polynomials in $\Q[a,b,c]_{g/2}$ such that $$\omega_1 R_j = S \omega_j$$ for all $j=2,\ldots,g$.
Using the expansions of $a,b,c,\omega_1,\ldots,\omega_g$ at $P$, we compute $\calR$ by linear algebra.
Thus we obtain an isomorphism $C \isom \PP^1$.

Under $C \isom \PP^1$, the function $h$ corresponds to some $f \in \Q(x)^\times$.
Now $X$ is birational to the curve $y^2=f(x)$.
Multiply $f$ by a square to make it a polynomial.
Remove square factors (by computing $\gcd(f,f')$, etc.) to make $f$ separable.
By Riemann--Hurwitz, $\deg f$ is $2g+1$ or $2g+2$.

\item \label{Step:FindH}
(For odd $g$, find $H$.)
Now assume that $g$ is odd.
Let $F,G$ be as in Step~\ref{Step:Writinghasratio}.  
We seek $H \in \Q[a,b,c]_{g+1}$ separable and $J \in \Q[a,b,c]_{(g+5)/2}$ such that $FG \equiv H J^2 \pmod{Q}$; 
then the rational function $h=F/G$ equals $HJ^2/G^2$ on $C \colon Q=0$, so the function field of the smooth projective curve $X' \colonequals \Proj \frac{\Q[a,b,c,y]}{(y^2-H,Q)}$ equals $\Q(C)(\sqrt{HJ^2/G^2}) = \Q(C)(\sqrt{h})$, so $X' \isom X$; that is, $H$ is as in the statement of the theorem.
We cannot simply factor $FG$ to find $H$ and $J$, since $\Q[a,b,c]/(Q)$ is not a UFD.
Instead we will decompose the zero locus $D \colonequals Z_C(FG) \in \Div C$ as $U+2V$ with $U,V$ effective divisors on $C$ and $U$ reduced.
First, choose $p \in \PP^2(\Q)$ not on any line connecting geometric points in $D$ and not on any line tangent to a geometric point in $D$; then the projection from $p$ restricts to a morphism $\nu \colon C \to \PP^1$ that is injective on the geometric points in $D$ and unramified at those points.
Write $\nu_* D = U' + 2 V'$ with $U',V'$ effective divisors on $\PP^1$ and $U'$ reduced, using factorization in the homogeneous coordinate ring of $\PP^1$.
Let $U = \nu^* U' \intersect D$ and $V = \nu^* V' \intersect D$;
then $D = U + 2V$ by choice of $\nu$.
We have $\deg U = \deg U' = 2g+2$, so $\deg V = \deg V' = g+5$.
By Riemann--Roch on $C$, an effective divisor of even degree $2d$ is the zero locus of a form in $\Q[a,b,c]_d$, unique up to scalar and modulo multiples of $Q$; in particular, there exist $H \in \Q[a,b,c]_{g+1}$ and $J \in \Q[a,b,c]_{(g+5)/2}$ with $Z_C(H)=U$ and $Z_C(J)=V$; we find explicit $H$ and $J$ by linear algebra.
Then $FG \equiv \alpha H J^2 \pmod{Q}$ for some $\alpha \in \Q^\times$.
Evaluate $F,G,H,J$ at some zero of $Q$ in $\Qbar^3$ to find $\alpha$, and replace $H$ by $\alpha H$ to get $FG \equiv H J^2 \pmod{Q}$.

If a rational point on $C$ is given, projection from it defines an isomorphism $C \to \PP^1$.
Find an equation $y^2=f(x)$ for $X$ as in the last paragraph of \eqref{Step:EvenGenus}.
\end{enumerate}

\section{Precision analysis}\label{S:Precision}

\begin{table}[h!]
\centering{
\scalebox{0.85}{
\begin{tabular}{c|ccll}
Object & Space & $\ord_P$ & Absolute error & Relative error \\ \hline
$\omega_j$ & $\HH^0(X,\Omega^1)$ & $[0,2g-2]$ & ${}+O(q^B) \, dq$ & $\cdot (1+O(q^{B-2g+2}))$ \\
$\omega_j'$ & $\HH^0(X_L,\Omega^1)$ & $[0,2g-2]$ & ${}+O(q^B) \, dq$ & $\cdot (1+O(q^{B-2g+2}))$ \\
$t$ & $L(C)$ & $[-2,2]$ & ${}+O(q^{B-2g})$ & $\cdot (1+O(q^{B-2g+2}))$ \\
$dt$ & $(\Omega^1_{C_L})_{\eta_{C_L}}$ & $[-3,1]$ & ${}+O(q^{B-2g-1}) \, dq$ & $\cdot(1+O(q^{B-2g-2}))$ \\
$t^i \frac{d}{dt}$ & $\HH^0(C_L,\calT)$ & $[-1,3]$ & ${}+O(q^{B-2g-3}) \, \frac{d}{dq}$ & $\cdot (1+O(q^{B-2g-2}))$ \\
$\del_i$ & $\HH^0(C,\calT)$ & $[-1,3]$  & ${}+O(q^{B-2g-3}) \, \frac{d}{dq}$ & $\cdot (1+O(q^{B-2g-6}))$ \\
$M(\del_0,\del_1,\del_2)$ & $\HH^0(C,\calT^{d})$ & $[-d,3d]$ & ${}+O(q^{B-2g-d-2}) \, \left(\frac{d}{dq}\right)^d$ & $\cdot (1+O(q^{B-2g-4d-2}))$ \\
$f$ & $\Q(C)$ & $[-4,4]$ & ${}+O(q^{B-2g-10})$ & $\cdot (1+O(q^{B-2g-6}))$ \\
$df$ & $(\Omega^1_C)_{\eta_C}$ & $[-5,5]$ & ${}+O(q^{B-2g-11})\,dq$ & $\cdot (1+O(q^{B-2g-16}))$ \\
$y$ & $\Q(X)$ & $[-2g-3,5]$ &${}+O(q^{B-4g-19})$ &$\cdot (1+O(q^{B-2g-16}))$ \\
$h$ &$\Q(C)$ & $[-4g-6,10]$ &${}+O(q^{B-6g-22})$ &$\cdot (1+O(q^{B-2g-16}))$ \\
$hG$ & $\calT^{g+3}_{\eta_{C}}$ & $[-5g-9,3g+19]$ & ${}+O(q^{B-11g-23}) \, \left(\frac{d}{dq}\right)^{g+3}$   & $\cdot (1 + O(q^{B-6g-14}))$\\
$F-hG$ & $\calT^{g+3}_{\eta_{C}}$ &  & ${}+O(q^{B-11g-23})\, \left(\frac{d}{dq}\right)^{g+3}$   & \\
$\omega_1 R_j$, $S\omega_j$ & $\HH^0(X,\Omega_X^1 \tensor \pi^\ast \calT^{\frac{g}{2}})$ & $[-g/2,7g/2-2]$ & ${}+O(q^{B-9g/2-2})\, \left(\frac{d}{dq}\right)^{\frac{g}{2}-1}$   & $\cdot (1+O(q^{B-4g-2}))$ \\
\end{tabular}
}
}
\caption{Tracking $q$-adic precision of objects in the proof of the main theorem.}
\label{table:error} % must be after caption
\end{table}

In Section~\ref{S:proof of main theorem}, we assumed that $w_1,\ldots,w_g \in K[[q]]$ were given to infinite precision.
Now, in \Cref{table:error}, we track how much precision we have in the steps if we start only with $w_1,\ldots,w_g$ up to addition of $O(q^B)$.
For each Laurent series, we bound both absolute error (addition of $O(q^n)$ for some $n$) and relative error (multiplication by $1+O(q^n)$ for some $n$); we can pass between them if the valuation of a power series is controlled; these valuations lie in the range given in the $\ord_P$ column of \Cref{table:error}.
For series with coefficients in the \'etale algebra $L \tensor K$, the bounds apply when projected onto any field factor of $L \tensor K$.
Let $\eta_C$ be the generic point of $C$, so the stalk $(\Omega^1_C)_{\eta_C}$ is the space of meromorphic $1$-forms on $C$.
Define $\eta_{C_L}$ similarly.

\begin{lemma}\label{Lemma:Precision}
\Cref{table:error} is correct.
\end{lemma}

\begin{proof}
Each $\omega_j$ is regular and has $2g-2$ zeros in total, so $\ord_P(\omega_j) \in [0,2g-2]$.
It is given to absolute error $O(q^B) \, dq$.
The $\omega_j'$ are $L$-linear combinations of the $\omega_j$, so they have the same absolute error.
Since each $\omega_j$ and $\omega_j'$ vanishes at $P$ to order at most $2g-2$, their relative error is  $1+O(q^{B-2g+2})$ (as usual, big-$O$ notation allows for the possibility that the error could be smaller than specified).

Now $t = \omega_{g-1}' / \omega_g'$, so its relative error is again $1+O(q^{B-2g+2})$.
On the other hand, $t$ is the ``$x$-coordinate'' of a hyperelliptic model of $X$, so $\ord_P(t) \in [-2,2]$.
Since $\ord_P t \ge -2$, the absolute error of $t$ is $O(q^{B-2g})$.

The absolute error of $dt$ is then $O(q^{B-2g-1}) \, dq$.
Again since $t$ is the ``$x$-coordinate'' of a hyperelliptic model of $X$, we have $\ord_P(dt) \in [-3,1]$.
Since $\ord_P(dt) \le 1$, the relative error of $dt$ is $1 + O(q^{B-2g-2})$.

Fix $i \in \{0,1,2\}$.
The relative error of $t^i \frac{d}{dt}$ is the worse of the relative errors of $t$ and $dt$, which is $1 + O(q^{B-2g-2})$.
The section $t^i \frac{d}{dt}$ is regular on $C$, with $2$ zeros, so $\ord_{\pi(P)}(t^i \frac{d}{dt}) \in [0,2]$, 
so $\ord_P (t^i \frac{d}{dt} )$ is in $[0,2]$ or $-1 + 2 [0,2] = [-1,3]$ according to whether $\pi$ is unramified or ramified at $P$, by \Cref{L:pullback of sections} applied with $n=-1$.
Since $\ord_P (t^i \frac{d}{dt} ) \ge -1$, the absolute error of $t^i \frac{d}{dt}$ is $O(q^{B-2g-3}) \, \frac{d}{dq}$.

The $\del_i$ are linear combinations of the $t^i \frac{d}{dt}$, so they have the same absolute error $O(q^{B-2g-3}) \, \frac{d}{dq}$.
As for $t^i \frac{d}{dt}$, we have $\ord_P (\del_i) \in [-1,3]$.
Since $\ord_P(\del_i) \le 3$, the relative error of $\del_i$ is $1 + O(q^{B-2g-6})$.

We will need to analyze the error in $M(\del_0,\del_1,\del_2)$ for various nonzero forms $M \in \HH^0(\PP^2,\OO(d)) = \Q[a,b,c]_d$, for various $d \ge 1$, so we do a calculation for all of these at once, and later specialize to the particular $M$ we need.
Its order of vanishing at $\pi(P)$ is in $[0,2d]$, since $C$ is a curve of degree $2$ in $\PP^2$.
Then its order of vanishing at $P$ is in $[0,2d]$ or $-d + 2[0,2d] = [-d,3d]$, according to whether $\pi$ is unramified or ramified at $P$, by \Cref{L:pullback of sections} applied with $n=-d$.
If $M$ is a monomial, then the absolute error of $M(\del_0,\del_1,\del_2)$ is at worst that of one $\del_i$ minus $d-1$ (because all the $\del_j$ in the monomial have at worst a simple pole), hence at worst $O(q^{B-2g-d-2}) \left( \frac{d}{dq} \right)^d$.
Since $\ord_P(M(\del_0,\del_1,\del_2)) \le 3d$, the relative error is at worst $1+O(q^{B-2g-4d-2})$.

The rational function $f \colonequals a/b$ is of degree $2$ on $C$, and the ramification index of $\pi$ at $P$ is at most $2$, so $\ord_P(f) \in [-4,4]$.
Its relative error is the same as that of $a=\del_0$ and $b=\del_1$, which is $1 + O(q^{B-2g-6})$.
Since $\ord_P(f) \ge -4$, its absolute error is $O(q^{B-2g-10})$.

Then $\ord_P(df) \ge \ord_P(f) - 1 \ge -5$.
Since $f$ on $C$ has at most $2$ poles with multiplicity,
$df$ on $C$ has at most $4$ poles with multiplicity,
but the divisor of $df$ on $C$ has degree $-2$, so $df$ has at most $2$ zeros on $C$, so $\ord_P(df) \le 2 \cdot 2 + 1 = 5$, the worst case being if $\pi$ is ramified at $P$.
The absolute error of $df$ is $O(q^{B-2g-11})$, so the relative error is $1+O(q^{B-2g-16})$.

Since $\ord_P(df) \in [-5,5]$ and $\ord_P(\omega_1) \in [0,2g-2]$,
we have $\ord_P(y) = \ord_P(df/\omega_1) \in [-2g-3,5]$.
The relative error of $y$ is the worse of the relative errors of $df$ and $\omega_1$, which is $1+O(q^{B-2g-16})$.
Then the absolute error of $y$ is $O(q^{B-4g-19})$.

Squaring gives $\ord_P(h)\in 2[-2g-3,5] = [-4g-6,10]$, and $h$ has relative error $1+O(q^{B-2g-16})$ and absolute error $1+O(q^{B-6g-22})$.

For $hG$, we compute $\ord_P$ and the relative error from the corresponding numbers for $h$ and $M \colonequals G$ of degree $d=g+3$.
Since $\ord_P(hG) \ge -5g-9$, the absolute error is then $O(q^{(B-6g-14) + (-5g-9)})\, \left(\frac{d}{dq}\right)^{g+3} = O(q^{B-11g-23})\, \left(\frac{d}{dq}\right)^{g+3}$.
The absolute error for $F$, from the $M(\del_0,\del_1,\del_2)$ row with $d=g+3$, is $O(q^{B-2g-(g+3)-2})\, \left(\frac{d}{dq}\right)^{g+3}$.
Combining these gives $F-hG$ with absolute error $O(q^{B-11g-23})\, \left(\frac{d}{dq}\right)^{g+3}$.
(We do not need the $\ord_P$ and relative error of $F-hG$.)

The calculations for $\omega_1 R_j$ and $S \omega_j$ are analogous to those for $hG$.
\end{proof}

\begin{lemma}
If $B \ge 19g + 48$, then we can perform Steps~\ref{Step:ConvenientBasis}--\ref{Step:EvenGenus} in the proof of \Cref{theorem:recovering}.
$($Step \ref{Step:FindH} does not involve expansions; it is carried out exactly.$)$
\end{lemma}

\begin{proof}
In Step~\ref{Step:FindingQ}, $Q$ is determined by $Q(\del_0,\del_1,\del_2) \bmod q^7$ because when $M \in \HH^0(\PP^2,\OO(2))$, we have $\ord_P M(\del_0,\del_1,\del_2) \le 6$.
We computed $Q(\del_0,\del_1,\del_2)$ to absolute error $O(q^{B-2g-d-2}) \left(\frac{d}{dq} \right)^d$ with $d=2$, which is good enough since $B-2g-2-2 \ge 7$.

Let $(h)_\infty$ be the polar part of the divisor of $h$, which has degree at most $2g+6$ as explained in Step~\ref{Step:Writinghasratio}.
For any $F,G \in \Q[a,b,c]_{g+3}$, the expression $F-hG$ is a global section of $\calT^{g+3} \tensor \OO_C((h)_\infty)$, which is a line bundle of degree at most $2(g+3) + (2g+6) = 4g+12$.
The pullback of this bundle to $X$ has degree at most $2(4g+12) = 8g+24$.
Thus, if $\ord_P(F-hG) > 8g+24$, then $F-hG=0$.
In other words, it suffices to do the linear algebra in Step~\ref{Step:Writinghasratio} to absolute precision $+O(q^{8g+25})\left(\frac{d}{dq}\right)^{g+3}$.
By \Cref{table:error}, we have this precision if $B-11g-23 \ge 8g+25$, or, equivalently, $B \ge 19g + 48$.

The degree of $\Omega^1_X \tensor \pi^* \calT^{g/2}$ on $X$ is $2g-2 + 2 \cdot 2 (g/2) = 4g-2$,
so if $\ord_P(\omega_1 R_j - S \omega_j) > 4g-2$, then $\omega_1 R_j - S \omega_j = 0$.
In other words, it suffices to do the linear algebra in Step~\ref{Step:EvenGenus} to absolute precision $+O(q^{4g-1})\left(\frac{d}{dq}\right)^{g/2-1}$.
By \Cref{table:error}, we have this precision if $B-9g/2-2 \ge 4g-1$, or, equivalently, $B \ge 17g/2 + 1$.
\end{proof}

\begin{remark}
In each of the models found, we have the expansions at $P$ of the new coordinate functions, as Laurent series in $q$, so we can find the coordinates of $P$ in the new model.
Similarly, by linear algebra we can express $\omega_1,\ldots,\omega_g$ in terms of the new coordinates, if desired. 
\end{remark}

\begin{remark}
The genus of a modular curve with geometric gonality $2$ is at most $17$ \cite[Remark~4.5]{baker-gonzalez-jimenez-gonzalez-poonen-05}.
So, in running the algorithm of \Cref{theorem:recovering} on hyperelliptic modular curves, we always have $g \le 17$.
\end{remark}

\section*{Acknowledgments} 
We thank David Roe for his comments on early versions of the algorithm of this article.
His computations using our algorithm were essential in helping us debug and improve it.

Van Bommel, Costa, and Poonen were supported by Simons Foundation grant 550033.
Van Bommel was additionally supported by C\'eline Maistret’s Royal Society Dorothy Hodgkin Fellowship.
Costa was additionally supported by Simons Foundation grant SFI-MPS-Infrastructure-00008651.
Poonen was partially supported also by National Science Foundation grant DMS-2101040 and Simons Foundation grant 402472.
Srinivasan was supported by National Science Foundation grant DMS-2401547 and Simons Foundation grant 546235.

\printbibliography

\end{document}